\documentclass[11pt]{article} 

\usepackage{geometry} 
\usepackage{booktabs} 
\usepackage{array} 
\usepackage{paralist} 
\usepackage{verbatim} 
\usepackage{subfig} 
\usepackage{amssymb,amsmath, amsthm} 
\usepackage{outlines}

\newcommand{\N}{\mathbb{N}}

\newcommand{\R}{\mathbb{R}}

\newcommand{\Pw}{\mathcal{P}}
\newcommand{\Inv}{\text{Inv}}
\newcommand{\Int}{\text{int}}
\newcommand{\cl}{\text{cl}}

\usepackage{fancyhdr} 
\usepackage{sectsty}
\allsectionsfont{\sffamily\mdseries\upshape} 

\usepackage[nottoc,notlof,notlot]{tocbibind} 
\usepackage[titles,subfigure]{tocloft} 

\usepackage{pdfpages}
\usepackage{graphicx}
\usepackage{subfig}

\newtheorem{theorem}{Theorem}[section]

\newtheorem{lemma}{Lemma}[section]
\newtheorem{exmp}{Example}[section]

\newtheorem{definition}{Definition}[section]

\title{Conley Index Theory and the Attractor-Repeller Decomposition for Differential Inclusions}
\author{Cameron Thieme}

\date{\today}

\begin{document}

\maketitle

\begin{abstract}
The Conley index theory is a powerful topological tool for describing the basic structure of dynamical systems.  One important feature of this theory is the attractor-repeller decomposition of isolated invariant sets.  In this decomposition, all points in the invariant set belong to the attractor, its associated dual repeller, or a connecting region.  In this connecting region, points tend towards the attractor in forwards time and the repeller in backwards time.  This decomposition is also, in a certain topological sense, stable under perturbation.  

Conley theory is well-developed for flows and homomorphisms, and has also been extended to some more abstract settings such as semiflows and relations.  In this paper we aim to extend the attractor-repeller decomposition, including its stability under perturbation, to continuous time set-valued dynamical systems.  The most common of these systems are differential inclusions such as Filippov systems.  
\end{abstract}
\let\thefootnote\relax
\let\thefootnote\relax
\section{Introduction}

In recent years, mathematicians and scientists have become increasingly interested in set-valued dynamical systems.  These systems are often described as differential inclusions 
\begin{equation}\label{dfnc}
\dot{x}\in F(x)
\end{equation} where $F:X\to\Pw(\R^n)$ is a set-valued map, $X$ is a subset of $\R^n$, and $\Pw(Y)$ denotes the power-set of the set $Y$.  
A solution of \eqref{dfnc} is an absolutely continuous function $$x:I\to\R^n$$ on some interval $I\subset \R$ whose derivative satisfies $\dot{x}(t)\in F(x(t))$ for almost all $t$ on the interval $I$.  

Differential inclusions are used extensively in control theory, and many phenomena, including friction and mechanical switching, are often modeled as differential inclusions \cite{bernardo}.  These inclusions also frequently arise in climate science \cite{welander}.

Because the set-valued nature of these systems allows for non-unique solutions, the behavior of differential inclusions can be very difficult to understand.  In this paper we will try to describe some of the qualitative features of these systems by extending some results from Conley index theory--a topological tool first developed for flows--to this setting.  In particular, we will extend the attractor-repeller decomposition of invariant sets to this more general setting.  For readers interested in the classical results of Conley index theory, \cite{mischaikow} provides an excellent introduction.  

In the classical setting, if $S$ is an invariant set for some flow $\varphi$, then we call $A\subset S$ an \textit{attractor in $S$} if $A$ is the $\omega$-limit set of a neighborhood of itself in $S$.  Associated to this attractor is a dual-repeller $R:=\{x\in S| \omega(x)\not\subset A\}$, and for all other points $x\in S$, $\alpha(x)\subset R$ and $\omega(x)\subset A$.  Moreover, this decomposition is stable in a topological sense; this attractor-repeller decomposition continues to nearby flows.  In this paper we will show that all of these results also hold for the solution set to \eqref{dfnc}.  The general decomposition is done in theorem \ref{att_rep_decomp} and the continuation result is theorem \ref{att_rep_cont_thm}.

In order to discuss these concepts, we will need to define $\omega$-limit sets, attractors, and dual repellers for the inclusion \eqref{dfnc}.  Other works (\cite{bhs}, \cite{oyama}, \cite{li}, \cite{kopanskii}, \cite{valero1}, \cite{valero2}) have already done similar work for certain classes of differential inclusions, but the notions adopted here are slightly different because of our distinct perspective on the solution set of \eqref{dfnc}.  Indeed, some of these works present similar theorems to those appearing in this paper, giving some kind of attractor-repeller decomposition for differential inclusions.  However, we choose to pursue our altered perspective because it aligns well with the Conley index; our ultimate aim is to place differential inclusions into a unified context where they can be completely analyzed using the ideas presented by Conley in \cite{conley}.  

In particular, we are concerned with the continuation of the decomposition to nearby systems.  This property has important implications for piecewise-continuous differential equations--which are generally reframed as differential inclusions(\cite{filippov})--because it allows us to make rigorous statements about certain families of smooth systems which limit to the discontinuous one (\cite{thieme2}).  Additionally, our setting is very general, allowing us to place minimal conditions, and no bounding term, on the set-valued map defining the inclusion \eqref{dfnc}.


This paper is split into five main sections.  The first section is this introduction, and the next one is a review of some results from differential inclusions.  The third section describes the attractor-repeller pair decomposition, including the new generalizations of the $\omega$-limit set.  We discuss the continuation of this decomposition in the following section.  The final section contains a few concluding remarks and acknowledges some valuable contributions.  

\section{The Basics of Differential Inclusions}\label{diffinc}

Before diving into the extension of Conley theory we will review some basic information about differential inclusions and state all of our assumptions on the set-valued map $F$ from \eqref{dfnc}.  We will also discuss the \textit{multiflow}, a set-valued analog of the flow that was introduced by Richard McGehee in order to describe the solution set to differential inclusions \cite{mcspeech}.  This set-valued map is distinct from earlier, similar objects (\cite{bhs}, \cite{oyama}, \cite{li}, \cite{kopanskii}, \cite{valero1}, \cite{valero2}) because it allows the image of a point to be empty in at some positive time, giving us a unique way to deal with finite-time blowup and allowing us to study a larger class of differential inclusions.  

\subsection{Differential Inclusions}

In order to analyze solutions to \eqref{dfnc}, we need to place a few assumptions on the map $F$.  The first condition is \textit{upper-semicontinuity}, and it is somewhat analogous to continuity in the single-valued setting.  

\begin{definition}
Let $X$ and $Z$ be metric spaces.  A set-valued function $F:X\to\Pw(Z)$ is said to be \textbf{upper-semicontinuous at the point x} if for any $\varepsilon>0$ there exists some $\delta>0$ such that $|x-y|<\delta$ implies that $F(y)\subset B_\varepsilon(F(x))$.
 
$F$ is said to be \textbf{upper-semicontinuous} if it is upper-semicontinuous at each $x\in X$.
\end{definition}

In addition to upper-semicontinuity, there are a few more properties of the set-valued map $F$ that we demand for the differential inclusion \eqref{dfnc}.  These requirements are described as the basic conditions by Filippov (\cite{filippov}), but we will give them a new name that reflects his role in this theory.

\begin{definition}
Let $X,Z$ be metric spaces.  The set-valued map $F:X\to\Pw(Z)$ is said to satisfy the \textbf{Filippov Conditions} if it is upper-semicontinuous and if the set $F(x_0)$ is
\begin{itemize}
\item Compact
\item Convex
\item Non-empty
\end{itemize}
for each $x_0\in X$.  
\end{definition}

Notice that the Filippov conditions do not place any sort of linear bounding term on the map $F$.  These will be the only assumptions that we place on our maps in this paper, making these results on limit sets somewhat more general than those found in \cite{oyama} or \cite{bhs}.

For the remainder of this paper, we will deal with differential inclusions defined in Euclidean space.  In \cite{filippov}, Filippov demonstrates that solutions to \eqref{dfnc} share several properties with the solutions to ordinary differential equations.  In particular, solutions in a compact domain are equicontinuous, and the limit of a uniformly convergent sequence of solutions is itself a solution.  These results mean that the solution set of \eqref{dfnc} behaves somewhat like a continuous flow $\varphi$.  Combining these results with the Arzela–Ascoli theorem, we can get the following lemma, which will be needed throughout our paper:

\begin{lemma}\label{extunisolbb}\cite{thieme2}
Assume that $X\subset\R^n$ is compact and that $F:X\to\R^n$ satisfies the Filippov conditions.  Given any sequence $\{x_k:\R\to X\}_{k=1}^\infty$ of solutions  to the differential inclusion $\dot{x}\in F(x)$, there is a solution $$x:\R\to X$$ to that inclusion such that on any compact interval $[a,b]\subset\R$, there is a subsequence of the restricted family $$\{x_k|_{[a,b]}:[a,b]\to X\}_{k=1}^\infty$$ which converges uniformly to $x|_{[a,b]}$.
\end{lemma}

\subsection{Multiflows: The Solution Set for Differential Inclusions}

Since differential inclusions do not have unique solutions---a given initial condition may be sent to infinitely many different locations at a fixed time $t$---their solution set can be very complicated.  Moreover, we cannot study the solution set with single-valued maps like flows.  To address this issue, Richard McGehee has proposed a different object, the \textit{multiflow}, which generalizes flows to this setting \cite{mcspeech}.  For a more complete exposition on this subject (and differential inclusions in general), the reader is referred to an earlier paper \cite{thieme}.

\begin{definition}
A \textbf{multiflow} on a metric space $X$ is a set-valued map $$\Phi:\R^+\times X\to\Pw(X)$$ which is upper-semicontinuous and compact-valued, and which satisfies the monoid properties:
\begin{itemize}
\item $\Phi(0,x)=\{x\}$
\item $\Phi(t,\Phi(s,x))=\Phi(t+s,x)$
\end{itemize}
\end{definition}
\noindent This definition relies on the notation that if $U\subset X$, $$\Phi(t,U)=\cup_{x\in U} \Phi(t,x)$$

In the remainder of the document, we will also utilize the notation that for $I\subset\R$, then $$\Phi(I,x)=\cup_{t\in I} \Phi(t,x),\hspace{1cm}\Phi(I,U)=\cup_{t\in I}\cup_{x\in U} \Phi(t,x)$$

The multiflow is a useful object because the set of all solutions to \eqref{dfnc} in any compact domain $X$ forms a multiflow:

\begin{theorem}\label{mtflw}\cite{thieme}
Let $G$ be an open subset of $\R^n$, let $X\subset G$ be compact, and let $F:G\to\Pw(\R^n)$ satisfy the Filippov conditions.  Define the set-valued map $$\Phi:\R^+\times X\to \Pw(X)$$ by saying that $b\in\Phi(T,a)$ if and only if there exists a solution $x:[0,T]\to X$ to the differential inclusion $\dot{x}\in F(X)$ with $x(0)=a$ and $x(T)=b$.

Then $\Phi$ is a multiflow over $X$.
\end{theorem}

It is also important to note that way we have defined the solution set here is somewhat unusual--the multiflow consists of solutions to the differential equation up until they leave the compact set.  That means it entirely possible that $\Phi(t,x)=\emptyset$ for some $x\in X$ and $t>0$.  For instance, if we consider the trivial differential inclusion $\dot{x}\in F(x)=1$ on the compact interval $X=[0,1]$, we see that $\Phi(t,x)=\emptyset$ for all $x$ and all $t>1-x$.  

Although this feature is relatively odd--the maximal solution for a given initial condition may exist on a compact interval--it is very helpful.  In fact, this perspective is what allows us to consider such a general class of differential inclusions.  Other work on set-valued dynamics (\cite{bhs}, \cite{oyama}, \cite{li}, \cite{kopanskii}, \cite{valero1}, \cite{valero2}) do not allow for $\Phi(t,x)$ to be empty, which places a restriction on the class of differential inclusions which can be considered.  Moreover, this perspective fits well with the idea of the Conley index theory, where global behavior can be understood using only information about the flow on the boundary of a compact set.  

One unfortunate drawback to multiflows is that, like semiflows, they consider only forward time.  This feature is a necessity; it is impossible to retain the group action of all of $\R$ \cite{thieme}.  So that we may consider solutions in backwards time we will introduce the dual multiflow $\Phi^*:\R^-\times X\to X$, which we define pointwise by $$\Phi^*(T,a):=\{b\in X|a\in\Phi(-T,b)\}$$

Essentially, we may think of this object as the backwards time equivalent of the multiflow $\Phi$.  If $\Phi$ arises from the differential inclusion $\dot{x}\in F(x)$, as in theorem \ref{mtflw}, then it is straightforward to verify that $\Phi^*$ is the set of all solutions to the same differential inclusion in backwards time.  That is, $$b\in\Phi^*(T,a)$$ if and only if there is a solution $$x:[T,0]\to X$$ satisfying $x(0)=a$ and $x(T)=b$.  

Finally, if the multiflow $\Phi$ arises from the differential inclusion $\dot{x}\in F(x)$ then we will call solutions of that differential inclusion \textbf{orbits} on $\Phi$.  Note that if an orbit $\psi$ has a domain that includes both positive and negative times then its image lies in both $\Phi$ and $\Phi^*$.  To be more specific about what that means, assume that $I\subset\R$ is an interval around $0$ and $\psi:I\to X$ satisfies $\dot{\psi}(t)\in F(\psi(t))$ for almost all $t\in I$.  If $T\in I$ is positive then $$\psi(T)\in\Phi(T,\psi(0))$$ and if $T\in I$ is negative then $$\psi(T)\in\Phi^*(T,\psi(0))$$

\section{Attractor-Repeller Decomposition}

For the remainder of this section we will assume that $\Phi: \R^+\times X\to X$ is the multiflow over the compact space $X\subset \R^n$ associated to the differential inclusion \eqref{dfnc}.  Our goal for this section will be to discuss the attractor-repeller decomposition of compact invariant sets of $\Phi$.  Before doing so, however, we should define specifically what is meant by invariant in this setting:

\begin{definition}
A set $S\subset X$ is called \textbf{invariant} under the multiflow $\Phi$ if for each $x\in S$ there exists an orbit $$\psi:\R\to S$$ on $\Phi$ with $\psi(0)=x$.
\end{definition}

We will often also consider the maximal invariant subset of a given set $U$ under the multiflow $\Phi$, which is $$\Inv(U,\Phi):=\{x\in U| \exists\,\text{orbit}\,\psi:\R\to U\,\text{on}\,\Phi,\,\psi(0)=x\}$$If the choice of multiflow $\Phi$ is clear, we will sometimes shorten this notation to $\Inv(U)$.  

Note that because we are working with set-valued systems, it is possible for points in an invariant set to have an orbit which leaves the invariant set.  That is, since a given point may have many different orbits, some of these orbits can leave the invariant set.  All that is required is that each point in the invariant set has \textit{at least one} orbit which stays in the set for all time.

This definition of invariance--which is sometimes called weak invariance--is not the only possible notion of invariance for differential inclusions.  We could, alternatively, demand that \textit{all} orbits stay in $S$ for all time, a condition that is sometimes called strong invariance.  It is this more restrictive notion of invariance that is analyzed in \cite{li}, which gives us a fairly different perspective.

\subsection{Limit Sets for Multiflows}

In order to discuss attractors and repellers, we need to define the concepts of limit sets for multiflows.

\begin{definition}
The $\omega$-limit set of a set $U$ is defined by $$\omega(U) = \cap_{t\geq 0} \overline{\Phi([t,\infty),U)}$$
\end{definition}

This definition of an $\omega$-limit set for multiflows is a direct generalization of the classical definition for flows.  Note that $\omega(U)$ is the set of all points $x\in X$ such that $$x = \lim_{n\to\infty} \psi_n(t_n)$$ where $\psi_n$ is an orbit on $\Phi$, $\psi_n(0)\in U$ and $t_n\to\infty$.  

As is the case with flows, we can also consider the $\alpha$-limit set, which is essentially the $\omega$-limit set in backwards time.  

\begin{definition}
The $\alpha$-limit set of a set $U$ is defined by $$\alpha(U) = \cap_{t<0}\overline{\Phi^*((-\infty,t],U)}$$
\end{definition}

Unfortunately, since some solutions can leave invariant sets, the $\omega$-limit set is not extremely well behaved.  For instance, if we have a multiflow over some space $X$ which is not itself invariant, it is possible that the $\omega$-limit set is also not invariant.  For instance, consider the differential inclusion $$\dot{x}\in F(x)=\begin{cases}
[0,1] & x= 0\\
1 & x \neq 0
\end{cases}$$Let $\Phi$ be the multiflow on the compact interval $X=[0,1]$ associated to this differential inclusion.  Then for $\Phi$, $\omega(0)=[0,1]=X$, but $X$ is not invariant.  Moreover, we note that $S=\{0\}$ is an invariant set for this multiflow, but $\omega(S)\not\subset S$.  

The fact that it is possible that $\omega(S)\not\subset S$ for an invariant set $S$ means that when we are trying to describe the dynamics on $S$, we need to take an extra step to restrict our view to $S$.  Therefore we will be concerned with the following object:

\begin{definition}
Let $\Phi:\R^+\times X\to \Pw(X)$ be a multiflow and let $S\subset X$.  Then $\Phi^S:\R^+\times S\to S$ is \textbf{the multiflow $\Phi$ restricted to $S$} and is defined pointwise by saying that $b\in\Phi^S(T,a)$ if and only if there exists an orbit $\psi:[0,T]\to S$ on $\Phi$ such that $\psi(0)=a$ and $\psi(T)=b$.
\end{definition}

Note in particular that the orbits here are followed only until they leave $S$, even if they return.  For instance, it is possible that there is some point $b\in S$ such that $b\in\Phi(T,a)$ for some $(T,a)\in \R^+ \times S$, but $b\not \in \Phi^S(T,a)$.

\begin{definition}
Let $S$ be a closed invariant set for the multiflow $\Phi$ and let $U\subset S$. Then the $\omega_S$ and $\alpha_S$ limit sets of $U$ are the sets $$\omega_S(U) = \cap_{t\geq 0} \overline{\Phi^S([t,\infty),U)}, \hspace{1cm}\alpha_S(U) = \cap_{t<0}\overline{(\Phi^S)^*((-\infty,t],U)}$$ 
\end{definition}

Similarly to the case of the general $\omega$-limit set, we notice that $\omega_S(U)$ is the set of all points $x\in S$ such that $$x = \lim_{n\to\infty} \psi_n(t_n)$$ where $\psi_n$ is an orbit on $\Phi^S$, $\psi_n(0)\in U$ and $t_n\to\infty$.  Since $S$ is invariant, however, we also may assume that $\psi_n$ is defined for all time.  Although this domain does not follow a priori from the definition--since the definition of invariance only requires that each point have an orbit which remains in $S$ for all time, and not that all orbits remain in $S$ for all time--we notice that we can always extend any orbit on $\Phi^S$ to a maximal orbit which exists for all time since $S$ is invariant.  Said another way, although the definition only directly implies that $\psi_n:I_n\to S$, where $I_n$ is an interval containing $[0,t_n]$, we can extend $\psi_n$ beyond this interval because $\psi_n$ evaluated at the endpoints of $I_n$ (or the limiting value, in the case where $I_n$ is not closed) is a point in $S$, and therefore there is some orbit which exists for all time at these points which we may append to $\psi_n$.  Therefore, when we consider the $\omega_S$-limit set, we only need to think about maximal orbits which are defined on all of the real line.  

Notice that in the special case where $\Phi$ is a flow, $\omega_S(U)=\omega(U)$ for all $U\subset S$, and so this object is simply a generalization of the classic $\omega$-limit set.  And as we will see in the following section, this object is much more well-behaved than the general $\omega$-limit set of a multiflow.  

\subsection{Attractor-Repeller Decomposition of Compact Invariant Sets}

For the remainder of this section, assume that $S\subset X$ is a compact invariant set for $\Phi$.  

\begin{lemma}\label{omg_inv}  
For $U\subset S$, both $\omega_S(U)$ and $\alpha_S(U)$ are non-empty and invariant.
\end{lemma}

\begin{proof}
We will prove the result for the $\omega_S$-limit set and the same result for the $\alpha_S$-limit set follows by symmetry.  

Let $\{\psi_n\}_{n=1}^\infty$ be any sequence of orbits such that $\psi_n(\R)\subset S$ and $\psi_n(0)\in U$.  We know that such a sequence must exist because $S$ is invariant (and we do not demand that the $\psi_n$ are unique).  Let $t_n\to\infty$ and consider the sequence of points $\{\psi_n(t_n)\}_{n=1}^\infty\subset S$.  Since $S$ is compact, this sequence must contain a convergent subsequence with some limit $x$.  By definition, $x\in\omega_S(U)$, and therefore $\omega_S(U)$ is non-empty.  

Now let $x\in\omega_S(U)$, so $x = \lim_{n\to\infty} \psi_n(t_n)$.  For each $s\in\R$, let $\gamma_n(s)=\psi_n(t_n+s)$.  By lemma \ref{extunisolbb}, there is some orbit $\gamma:\R\to S$ such that on any compact interval $[a,b]$, the family $\{\gamma_n\}_{n=1}^\infty$ has some subsequence which converges uniformly to $\gamma$.  We see that $x=\gamma(0)$ and that $\gamma(s)=\lim_{n_k\to\infty}\gamma_{n_k}(s)=\lim_{n_k\to\infty}\psi_{n_k}(t_{n_k}+s)\in\omega_S(U)$ (for any given $s$ we can take $[a,b]$ to be large enough that we get the subsequence in that previous equality).  
\end{proof}

With the $\alpha_S$ and $\omega_S$ limit sets defined, we are now ready to define attractors and repellers for multiflows.  As in the traditional setting, an attractor is a set which is the $\omega_S$-limit set of a neighborhood of itself, and a repeller is a set which is the $\alpha_S$-limit set of some neighborhood of itself.  
\begin{definition}
A set $A\subset S$ is said to be an \textbf{Attractor in S} if there is a neighborhood $U$ of $A$ in $S$ such that $\omega_S(U)=A$.

A set $R\subset S$ is said to be an \textbf{repeller in S} if there is a neighborhood $U^*$ of $R$ in $S$ such that $\alpha_S(U^*)=R$.

\end{definition}

A crucial aspect of the attractor-repeller decomposition is that for a given attractor we can associate a specific dual repeller.  Symmetrically, if we begin with a repeller, we can associate a specific dual attractor.  

\begin{definition}
If $A\subset S$ is an attractor in $S$, then the \textbf{dual repeller of A in S} is the set $$R=\{x\in S| \omega_S(x)\not\subset A\}$$

If $R\subset S$ is a repeller in $S$, then the \textbf{dual attractor of R in S} is the set $$A=\{x \in S| \alpha_S(x)\not \subset R\}$$  

\end{definition}

At this point, it is unclear that the dual repeller is actually a repeller, or that the dual attractor is actually an attractor.  Moreover, it is unclear that the term \textit{dual} is justified--that the dual of the dual object is the original object.  However, we will ultimately see that this terminology is justified in theorem \ref{att_rep_decomp}.  This symmetry also distinguishes our work from earlier work extending the attractor-repeller decomposition to differential inclusions in \cite{li}.  In that paper, the repeller is not defined as an object in its own right, and only a definition of dual repeller is given.  By considering the notion of a repeller as its own object we see that more of the structure of the attractor-repeller decomposition carries over to this setting than was previously shown.  Before diving more into this structure, however, we should make a few remarks about our definition.

In the classical theory of Conley index for flows, the dual repeller is defined as the set $\{x\in S|\omega(x)\cap A=\emptyset\}$.  A simple lemma then shows that $\omega(x)\cap A\neq\emptyset$ if and only if $\omega(x)\subset A$, and so the definition that we have provided here for multiflows does indeed generalize the traditional definition.  However, in the case of differential inclusions it is possible for the $\omega_S$-limit set of a point to intersect an attractor without being a subset of that attractor, motivating our definition.  To see this phenomenon consider the following example:

\begin{exmp}\label{bad_dual_def}
Let $F:[-1,1]\to \R$ be defined by
$$F(x)=\begin{cases}
0 & x\in [-1,0)\\
[0,1] & x = 0\\
1-x & x\in(0,1]
\end{cases}$$

Note that $F$ satisfies the Filippov conditions, and let $\Phi$ be the associated multiflow.  Notice that $S=[-1,1]$ is invariant.  Then $A=\{1\}$ is an attractor in $S$.  Note that $\omega_S(0)=[0,1]$, and so $\omega_S(0)\cap A\neq\emptyset$ and $\omega_S(0)\not\subset A$. 

Also, notice that the set $$\{x\in [-1,1]|\omega_S(x)\cap A=\emptyset\}=[-1,0)$$ is not a repeller, but the set $$\{x\in [-1,1]|\omega_S(x)\not\subset A\}=[-1,0]$$ is a repeller (it is the dual-repeller to $A$).
\end{exmp}

\begin{figure}[h]
\begin{center}
\includegraphics[scale=0.4]{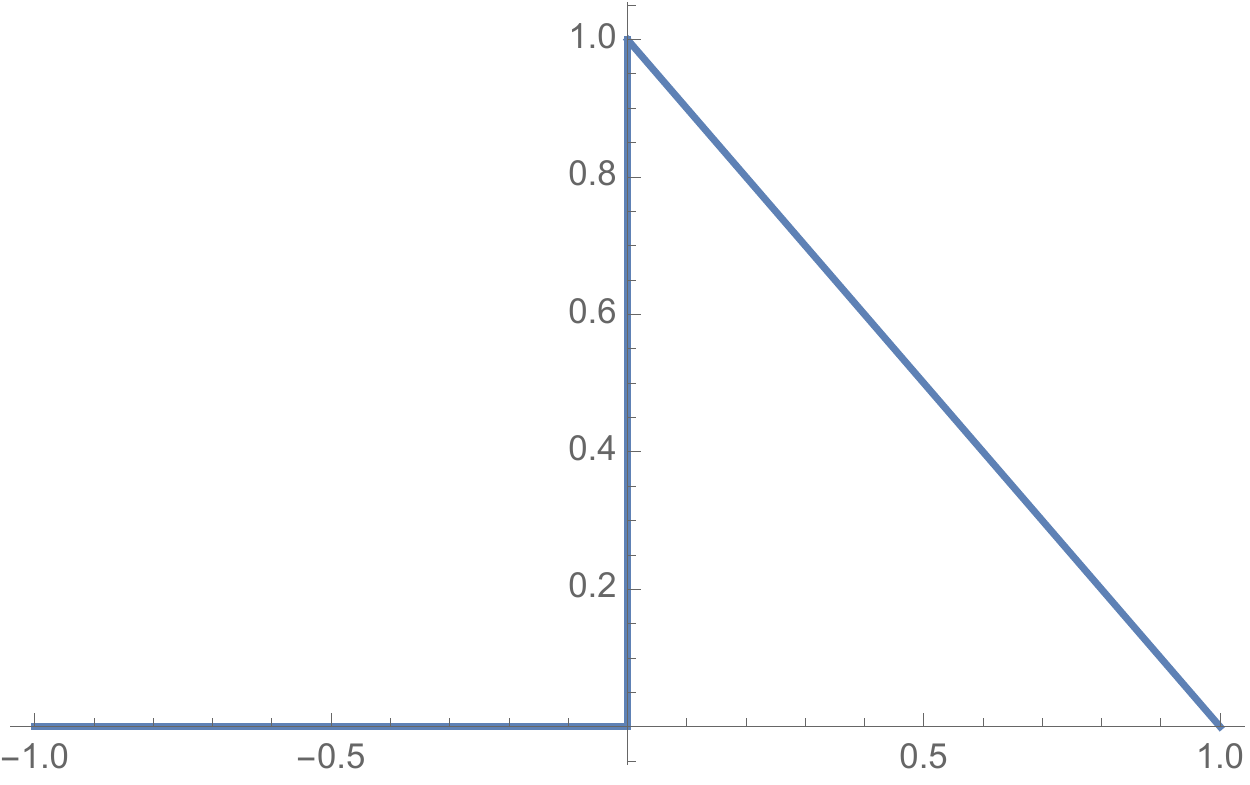}
\caption{The set-valued map $F$ from example \ref{bad_dual_def}}
\end{center}
\end{figure}

We also note that if $A$ is an attractor and $R$ is the associated dual-repeller, then $A\cap R=\emptyset$.  Therefore we can complete the decomposition of $S$ by simply taking the set of all remaining points, which we will call the connecting region.  

\begin{definition}
Given an attractor $A$ and its dual repeller $R$, define the \textbf{connecting region between $A$ and $R$} as $$C(A,R):=(A\cup R)^c$$
\end{definition}

Given these definitions, we see that $S=A\cup R\cup C(A,R)$.  We will call the pair $(A,R)$ an \textbf{attractor-repeller pair decomposition of the invariant set $S$}, and we will list its properties in theorem \ref{att_rep_decomp}.  Before stating and proving that theorem, however, we need the following lemma:

\begin{lemma}\label{contained_inv}
For any $U\subset S$, if $\omega_S(U)\subset U$ then $\Inv(U)=\omega_S(U)$.\\
Symmetrically, if $\alpha_S(U)\subset U$ then $\text{Inv}(U)=\alpha_S(U)$.
\end{lemma}

\begin{proof}
The inclusion $\omega_S(U)\subset\text{Inv}(U)$ follows from lemma \ref{omg_inv}.  Therefore we must only show that $\text{Inv}(U)\subset\omega_S(U)$.\\
Let $x\in\Inv(U)$.  Then by definition there is some orbit $\psi:\R\to U$ such that $\psi(0)=x$.  Given a sequence $\{t_n\}\to\infty$, $\psi(-t_n)\in U$, so $x\in\Phi^S(t_n,\psi(-t_n))$ for each $n\in\N$, and so $x\in\omega_S(U)$.
\end{proof}

We are now ready for one of the main theorems of this paper.  

\begin{theorem}\label{att_rep_decomp}
Let $\Phi:\R^+\times X\to\Pw(X)$ be a multiflow over a compact space $X\subset\R^n$ and assume that $S\subset X$ is invariant under $\Phi$.  Let $A$ be an attractor in $S$, $R$ its dual repeller, and $C(A,R)$ the connecting region between them.
\begin{enumerate}
\item $S=A\cup R\cup C(A,R)$ and the sets $A$, $R$ and $C(A,R)$ are all disjoint.\label{decomp_S}
\item $R$ is a repeller in $S$.\label{rep_justified}
\item $C(A,R)=\{x\in S| \,\omega_S(x)\subset A,\,\alpha_S(x)\subset R\}$.\label{conn_reg}
\item $A$ is the dual attractor to $R$. \label{dualdual}
\end{enumerate}
\end{theorem}

\begin{proof}
Item \ref{decomp_S} follows directly from the definitions of the relevant sets and is included only for emphasis.  Item \ref{dualdual} follows directly after proving items \ref{rep_justified} and \ref{conn_reg}.\\

\noindent Proof of \ref{rep_justified}:\\

Let $U$ be a neighborhood of $A$ in $S$ such that $\omega_S(U)=A$.  Then there is some time $t^*>0$ such that $\overline{\Phi^S([t^*,\infty),U)}\subset U$; if that were not the case, we could find a sequence of image points of $U$ whose limit was not in $A$.  Define $U^*:= S\setminus \overline{\Phi^S([t^*,\infty),U)}$; note that $S=U\cup U^*$.  We will show that $R=\alpha_S(U^*)$.  

We can see that $(\Phi^S)^*((-\infty,-t^*],U^*)\subset S\setminus U\subset U^*$.  If not, there would be some points $y\in U^*$ and $x\in U$ and some time $\tau>t^*$ such that $x\in (\Phi^S)^*(-\tau,y)$.  But then we would have that $y\in\Phi^S(\tau,x)$, contradicting our assumption on $t^*$.  From this inclusion it follows that $U^*$ is a neighborhood of $\alpha_S(U^*)$, and so by lemma \ref{contained_inv}, $\Inv(U^*)=\alpha_S(U^*)$.  Therefore we can show that $R\subset\alpha_S(U^*)$ by showing that $R\subset\Inv(U^*)$.  

We want to show that if $x\in R$ then there is some orbit with initial condition $x$ that remains in $U^*$ for all time.  Since $S$ is invariant, there is some orbit $\psi:\R\to S$ with $\psi(0)=x$.  Note that if \textit{every} orbit originating at $x$ had to enter $U$ in positive time then it follows that $\omega_S(x)\subset\omega_S(U)$ and so $x\not\in R$.  Therefore, without loss of generality, we can assume that $\psi(\R^+)\cap U=\emptyset$.  Now, if $\psi(-t)\in U$ for any $t>0$ then $\omega_S(x)\subset\omega_S(\psi(-t))\subset A$, contradicting the assumption that $x\in R$.  Thus $\psi(\R^-)\subset U^*$.  Therefore $x\in \Inv(U^*)=\alpha_S(U^*)$ and so $R\subset \alpha_S(U^*)$.

To see that $\alpha_S(U^*)\subset R$, we start by noting that if $x\in\alpha_S(U^*)$ then $\omega_S(x)\cap\alpha_S(U^*)\neq\emptyset$ by the invariance of the $\alpha_S$-limit set.  Since $\alpha_S(U^*)\subset U^*$ and $U^*\cap A=\emptyset$, we conclude $\omega_S(x)\not\subset A$ and so $x\in R$.\\

\noindent Proof of \ref{conn_reg}:\\

It follows directly from the definition that $\omega_S(x)\subset A$ for $x\in S\setminus R$, so we must only show that $\alpha_S(x)\subset R$ for $x\in S\setminus A$.  

Let $x\in S\setminus A$, and call $\delta:=\text{dist}(x,A)$.  Let $U'$ be a neighborhood of $A$ in $S$ such that $\omega_S(U')=A$.  Then $U:=U'\cap B_{\delta/2}(A)$ also satisfies $\omega_S(U)=A$.  As shown in the proof of part \ref{rep_justified}, there is some time $t^*>0$ such that $\overline{\Phi^S([t^*,\infty),U)}\subset U$ and $U^*:=S\setminus\overline{\Phi^S([t^*,\infty),U)}$ satisfies $\alpha_S(U^*)=R$.  Then since $x\in U^*$, $\alpha_S(x)\subset R$.
\end{proof}

We will close out this section with a lemma that gives a way to more easily identify attractors.  If the multiflow moves the closure of a set into its interior at some positive time then the set's $\omega_S$-limit set is an attractor.  This condition is very helpful in identifying attractors because it relies only on checking a single positive time.  

\begin{lemma}\label{1timerule}
Suppose $U\subset S$ and $\Phi^S(t,\overline{U})\subset\Int(U)$ for some $t>0$.  Then $\omega_S(U)$ is an attractor contained in the interior of $U$.  
\end{lemma}

\begin{proof}
Since $\Phi^S(t,\overline{U})\subset\Int(U)$, there is an open set $V$ such that $\Phi^S(t,\overline{U})\subset V\subset\overline{V}\subset\Int(U)$.  Then there is some $\varepsilon>0$ such that $\Phi^S((t-\varepsilon,t+\varepsilon),\overline{U})\subset V$.  If that were not the case then there would be a sequence  of times $t_n\to t$ and associated points $x_n\in \overline{U}$ and orbits $\psi_n$ with $\psi_n(0)=x_n\in \overline{U}$ and $\psi_n(t_n)\not\subset V$.  By lemma \ref{extunisolbb}, there is an orbit $\psi$ such that on any compact interval $I\subset \R$ there is some sequence $\{n_k\}_{k=1}^\infty$ where $\psi_{n_k}|_I\to\psi|_I$.  But $\psi(0)=\lim_{k\to\infty}\psi_{n_k}(0)\in \overline{U}$ and $\psi(t)\lim_{k\to\infty}\psi_{n_k}(t)\in V^c$, contradicting our assumption that $\Phi^S(t,\overline{U})\subset V$.

Then if $t'>t^2/\varepsilon$, we can write $t'=s_1+\cdots+s_m$ where $s_i\in(t-\varepsilon,t+\varepsilon)$.  Then $$\Phi^S(t',\overline{U})=\Phi^S(s_1+\cdots+s_m,\overline{U})=\Phi^S(s_m,\Phi^S(s_{m-1}(\cdots(\Phi^S(s_2,\Phi^S(s_1,\overline{U})))\cdots)))\subset V$$  We can see this by noticing that $\Phi^S(s_i,\overline{U})\subset\Phi^S((t-\varepsilon,t+\varepsilon),\overline{U})\subset V\subset\Int(U)$, and hence $\Phi^S(s_j,\Phi^S(s_i,\overline{U}))\subset\Phi^S((t-\varepsilon,t+\varepsilon),\overline{U})$.  Then for $t'>t^2/\varepsilon$, $\Phi([t',\infty),\overline{U})\subset\overline{V}\subset\Int(U)$.  Therefore $\omega_S(U)\subset\Int(U)$ and so $\omega_S(U)$ is an attractor.
\end{proof}

\section{Continuation of the Decomposition}

The most important thing about Conley Theory is that the information that it gives is, in some sense, stable under perturbation.  In this section we will briefly review what perturbation means in the setting of differential inclusions, but for a more complete explanation the reader is referred to \cite{thieme2}.  Our main goal for this section will be to make explicit the sense in which the attractor repeller decomposition is stable under perturbation, and also to prove that result.  

\subsection{Perturbed Solutions}

We will begin by defining the perturbation of the set-valued map of a differential inclusion.  The most important feature of this definition is that it directly generalizes the continuous perturbation of a single-valued function.  

\begin{definition}
Let $G\subset\R^n$ and assume that $F:G\times[-1,1]\to\R^n$ meets the Filippov conditions; in particular, $F$ is upper-semicontinuous in both $x$ and $\lambda$ together.  Then the differential inclusion $$\dot{x}\in F(x,\lambda)$$ is considered to be a \textbf{$\lambda$-perturbation} of the differential inclusion $$\dot{x}\in F(x,0)$$
\end{definition}

For this paper, the most important feature of perturbation is that convergent sequences of perturbed solutions converge to solutions of the original differential inclusion, as stated formally in the following lemma:

\begin{lemma}\label{extunisol}\cite{thieme2}
Assume that $X\subset\R^n$ is compact and that $F:X\times[-1,1]\to\R^n$ satisfies the Filippov conditions.  Let $\lambda_k\to 0$ as $k\to\infty$.  Given any sequence $\{x_k:\R\to X\}_{k=1}^\infty$ of  solutions  to the differential inclusion $\dot{x}\in F(x,\lambda_k)$, there is a solution $$x:\R\to X$$ to the inclusion $\dot{x}\in F(x,0)$ such that on any compact interval $[a,b]\subset\R$, there is a subsequence of the restricted family $$\{x_k|_{[a,b]}:[a,b]\to X\}_{k=1}^\infty$$ which converges uniformly to $x|_{[a,b]}$.
\end{lemma}

\subsection{Isolating Neighborhoods and Isolated Invariant Sets} 

The most basic objects of Conley index theory are the isolating neighborhood and the associated isolated invariant set.  
\begin{definition}
Let $\Phi:\R^+\times X\to X$ be a multiflow.  A compact set $N\subset X$ is called an \textbf{isolating neighborhood for $\Phi$} if its maximal invariant set lies in its interior; that is,
$$\text{Inv}(N;\Phi)\cap\partial N=\emptyset$$
A set $S\subset \R^n$ is called an \textbf{isolated invariant set for $\Phi$} if it is the maximal invariant set in some isolating neighborhood.  That is, $S$ is an isolated invariant set if there is an isolating neighborhood $N$ such that $$S = \text{Inv}(N;\Phi)$$
\end{definition}

Isolated invariant sets are the primary objects that we are concerned with qualitatively describing using the Conley index theory.  In the section describing the attractor-repeller decomposition, we assumed that the space we were decomposing was compact and invariant.  As the following lemma shows, isolated invariant sets meet those criteria, allowing us to use that decomposition in order to describe the dynamics of these sets.  

\begin{lemma}\cite{thieme2}
Isolated invariant sets of a multiflow are compact.
\end{lemma}

Additionally, attractors in isolated invariant sets are themselves isolated invariant sets.  This fact will be very important when we start discussing perturbation.  

\begin{lemma}\label{att_iso}
Let $S$ be an isolated invariant set for the multiflow $\Phi$.  If $A\subset S$ is an attractor in $S$, then $A$ is an isolated invariant set for the multiflow $\Phi$.  Symmetrically, a repeller in $S$ is an isolated invariant set.
\end{lemma}

\begin{proof}
Let $N$ be an isolating neighborhood for $S$ and let $U'\subset S$ be a neighborhood of $A$ in $S$ such that $\omega_S(U')=A$. Let $d_h(W,Z)$ denote the Hausdorff distance of the compact sets $W$ and $Z$, and define $$\delta:=\min(d_H(A,\partial N),d_H(A,\partial U'))$$Now let $U=U'\cup B_{\delta/2}(A)$.  We will show that $\overline{U}$ is an isolating neighborhood for $A$ in $\Phi$.

We need to show that $\Inv(\overline{U})\subset\Int(U)$, so let $x\in\partial U$. If $x\not\in S$, then we know that $x\not\in\Inv(\overline{U})$ because $S=\Inv(N)$ and $\overline{U}\subset \Int(N)$.  If $x\in S$, then $\alpha_S(x)\subset R$, where $R$ is the dual-repeller of $A$ in $S$.  Since $R\cap U'=\emptyset$, there cannot be an orbit with initial condition $x$ that remains in $\overline{U}$ for all time, and so $x\not\in\Inv(\overline{U})$.

\end{proof}

The most important property of isolating neighborhoods is that they are stable under perturbation.  This fact forms the basis of Conley Index theory, and it was extended to the case of multiflows in \cite{thieme2}.  

In order to state this result, we will assume that the set $X\subset\R^n$ is compact and that $F:X\times[-1,1]\to\Pw(\R^n)$ satisfies the Filippov conditions.  We will use this set-valued map in order to define a family of multiflows.  Define $$\Phi_\lambda:\R^+\times X\to \Pw(X)$$ by saying that $b\in\Phi_\lambda(T,a)$ if and only if there is a solution $x:[0,T]\to X$ of the differential inclusion $\dot{x}\in F(x,\lambda)$ such that $x(0)=a$ and $x(T)=b$.  We will carry these assumptions on the object $F$ and the family of multiflows $\Phi_\lambda$ for the duration of this section.  

It is worth noting here that this notion of perturbation is fairly general.  For instance, it allows us to consider perturbing the so-called Filippov systems--differential inclusions that arise from piecewise continuous differential equations--to smooth systems which limit to the Filippov system.  For more information on the nature and motivation of this sense of perturbation for set-valued dynamical systems, see \cite{thieme2}.

\begin{theorem}\label{pert_thm}\cite{thieme2}
If $N$ is an isolating neighborhood for the multiflow $\Phi_0$ then there exists some $\varepsilon>0$ such that $|\lambda|<\varepsilon$ implies that $N$ is an isolating neighborhood for $\Phi_\lambda$.  
\end{theorem}

\subsection{Attractor-Repeller Pair Continuation}

The ultimate goal of the Conley Index theory is to obtain results which are stable under perturbation.  Therefore we would like to show that, in some sense, our attractor-repeller pairs are stable up to perturbation of the differential inclusion.  To start that process we will define the continuation of isolated invariant sets.  

\begin{definition}
Let $N\subset X$ be a compact neighborhood, and denote $S_\lambda:=\Inv(N,\Phi_\lambda)$.  Two isolated invariant set $S_{\lambda_0}$ and $S_{\lambda_1}$ are \textbf{related by continuation} or \textbf{$S_{\lambda_0}$ continues to $S_{\lambda_1}$} if $N$ is an isolating neighborhood for all $\Phi_\lambda$, $\lambda\in[-\lambda_0,\lambda_1]\subset[-1,1]$.
\end{definition}

Note that this definition is exactly the same as the definition given in classical Conley Index theory, once the notion of invariance and perturbation has been understood.  Then, as in the classical case, it is worth mentioning here that continuation says nothing explicitly about the invariant sets $S_\lambda$, and is only a statement about isolating neighborhoods.  Indeed, the structure of the invariant sets is allowed to change somewhat drastically while remaining related by continuation.  For instance, a degenerate fixed point is often continued to the empty set.  For a simple example, consider the family of differential equations $\dot{x}=x^2+\lambda$.  Then the interval $[-1,1]$ is an isolating neighborhood for all $\lambda\in[0,1]$, and therefore $S_0=\{0\}$ continues to $S_1=\emptyset$.

This property of continuation is actually a feature of Conley theory and not a bug, allowing us to avoid the complications of bifurcation theory.  By using the Conley index, we can use knowledge of the behavior on the boundary of the isolating neighborhoods to obtain topological information about the associated isolated invariant sets.  At this point in time, the Conley index itself has not been generalized to differential inclusions and multiflows, but the results of this paper lead us to believe that this generalization is possible.

One interesting remark about the continuation of isolated invariant sets is that the invariant sets only change semicontinuously.  That is, isolated invariant sets which are related by continuation may suddenly shrink, as the example involving the degenerate fixed point and the emptyset shows, but they can only grow in a continuous way.  We see this by noticing that if $S$ is an isolated invariant set, then for arbitrarily small $\delta$, the set $\overline{B_\delta(S)}$ is an isolating neighborhood for $S$.  Since this isolating neighborhood is stable under perturbation, the continuation of $S$ is a subset of $\overline{B_\delta(S)}$ for sufficiently small perturbations of the multiflow.  Then because $\delta$ can be made arbitrarily small, it is clear that $S$ cannot grow discontinuously.  Since this result is used in the proof of one of our main theorems, we will state it formally as the following lemma.

\begin{lemma}\label{inv_semi}
Let $S_0$ continue to $S_\lambda$ for $\lambda\in I$, where $I$ is a closed interval around $0$.  Then if $\lambda_n\to0$ and $x_n\in S_{\lambda_n}$, then any convergent subsequence of $\{x_n\}_{n=1}^\infty$ must limit to a point in $S_0$.  
\end{lemma}

With this lemma stated, we are ready to prove one of the key results of this paper, showing that the attractor-repeller decomposition described in the prior section is stable under perturbation.  

\begin{theorem}\label{att_rep_cont_thm}
Attractor-repeller pair decompositions continue.  
\end{theorem}

This result is similar to a result found in \cite{li}, but the concept of invariance used here is different and the notion of perturbation is somewhat more general.  

In the proof of this result, we will need to discuss the interiors and closures of sets relative to other sets.  To do so, we will adopt the convention that $\Int(W;Z)$ and $\cl(W;Z)$ respectively denote the interior and closure of the set $W$ relative to $Z$.  Additionally, the notation $W\setminus Z$ does not imply here that $Z\subset W$, but merely is intended to convey the notion $W\setminus (W\cap Z)$.  

Additionally, before beginning this proof, we should acknowledge the role Richard Moeckel played in its development.  He offered some extremely valuable insights into the nature of continuation that come into play here.  

\begin{proof}

Assume that $S_0$ is an isolated invariant set for the multiflow $\Phi_0$ with isolating neighborhood $N$, and let $(A_0,R_0)$ be an attractor-repeller pair decomposition of $S_0$.  Then since $A_0$ and $R_0$ are themselves isolated invariant sets for $\Phi_0$ by lemma \ref{att_iso}, they have isolating neighborhoods $N_A\subset N$ and $N_R\subset N$.  Since isolating neighborhoods are stable under perturbation by lemma \ref{pert_thm}, there is some $\lambda_S>0$ such that $|\lambda|\leq\lambda_S$ implies that $N$ is an isolating neighborhood for $\Phi_\lambda$.  Similarly, there exist  $\lambda_A>0$ and $\lambda_R>0$ such that $N_A$ and $N_R$ remain isolating neighborhoods for $\lambda\in[-\lambda_A,\lambda_A]$ and $\lambda\in[-\lambda_R,\lambda_R]$ respectively.

Let $\lambda_0:=\min(\lambda_S,\lambda_A,\lambda_R)$.  Then for $\lambda\in[-\lambda_0,\lambda_0]$, the isolated invariant sets $A_\lambda:=\Inv(N_A,\Phi_\lambda)$ are related by continuation, the $R_\lambda:=\Inv(N_R,\Phi_\lambda)$ are related by continuation, and the $S_\lambda:=\Inv(N,\Phi_\lambda)$ are related by continuation.  Thus all that remains to check is that $(A_\lambda,R_\lambda)$ is an attractor-repeller pair decomposition for $S_\lambda$ for for sufficiently small $|\lambda|$.

We will start by showing that $A_\lambda$ is an attractor in $S_\lambda$ for small enough $|\lambda|$.  We know that is some time $t^*$ such that $$\Phi_0^{S_0}(t^*, N_A\cap S_0)\subset \Int(N_A\cap S_0; S_0)$$ since $A_0$ is assumed to be an attractor in $S_0$.  For a sufficiently small $|\lambda|$, we will show that $$\Phi^{S_\lambda}_\lambda(t^*, N_A\cap S_\lambda)\subset\Int(N_A\cap S_\lambda; S_\lambda)$$ which implies that $\omega_{S_\lambda}(N_A\cap S_\lambda)$ is an attractor in $S_\lambda$ by lemma \ref{1timerule}.  

If this were not the case, we would have a sequence $\lambda_n\to 0$ and associated points $x_n\in N_A\cap S_{\lambda_n}$ and orbits $\psi_n$ on $\Phi_{\lambda_n}$ satisfying $$\psi_n(0)=x_n\in N_A\cap S_{\lambda_n},\hspace{1cm}\psi_n(t)\not\in\Int(N_A\cap S_{\lambda_n};S_{\lambda_n})$$  By lemma \ref{extunisol}, on a compact interval $I$ containing $0$ and $t$, we can take some subsequence $\{n_k\}_{k=1}^\infty$ of these orbits which converge uniformly to an orbit $\psi$ on $\Phi_0$.  Notice that $$x_{n_k}=\psi_{n_k}(0)\to \psi(0)\in N_A\cap S_0$$ by lemma \ref{inv_semi}.  But lemma \ref{inv_semi} also shows that $\psi(t)\in \cl(S_0\setminus N_A; S_0)$, contradicting our assumption that $\Phi^{S_0}_0(t, N_A\cap S_0)\subset \Int(N_A\cap S_0)$.  Therefore $\omega_{S_\lambda}(N_A\cap S_\lambda)$ is an attractor in $S_\lambda$ for small enough $|\lambda|$.  Since $\omega_{S_\lambda}(N_A\cap S_\lambda)=A_\lambda$ by lemma \ref{contained_inv}, we see that $A_\lambda$ is an attractor as desired.  

We can follow a symmetric argument to see that $R_\lambda$ is a repeller in $S_\lambda$ for small enough $|\lambda|$, and so it only remains to show that $R_\lambda$ is the dual-repeller to $A_\lambda$ in $S_\lambda$.  That is, we must show that $\omega_{S_\lambda}(x)\subset A_\lambda$ for all $x\in S_\lambda\setminus R_\lambda$ for small enough $|\lambda|$.  

In fact, we actually only need to show this property for for $x\in S_\lambda\setminus N_R$.  This restriction is possible because $x\in (S_\lambda\cap N_R)\setminus R_\lambda$ implies that for \textit{any} orbit $\psi$ on $\Phi_\lambda^{S_\lambda}$ such that $\psi(0)=x$, there is some time $t$ such that $\psi(t)\in S_\lambda\setminus N_R$ because $N_R$ is an isolating neighborhood and $S_\lambda$ is an invariant set.  Then if $\omega_{S_\lambda}(x)\not\subset A_\lambda$, then also $\omega_{S_\lambda}(\psi(t))\not\subset A_\lambda$ for some such orbit.  

For the sake of contradiction, assume that this is not the case and $\omega_{S_\lambda}(x)\not\subset A_\lambda$ for all $x\in S_\lambda\setminus N_R$ for small enough $|\lambda|$.  Then there is some sequence $\lambda_n\to 0$ and associated points $x_n\in  S_{\lambda_n}\setminus N_R$ and $y_n\not\in A_{\lambda_n}$ such that $y_n\in \omega_{S_{\lambda_n}}(x_n)$. By the definition of the $\omega$-limit set, that means that for each $n$ there is a sequence of orbits $\{\psi_n^k\}_{k=1}^\infty$ and a sequence of times $t_n^k\to\infty$ such that $$\psi_n^k(0)=x_n,\hspace{1cm}\psi_{n}^k(t_n^k)\to y_n,\hspace{1cm}k\to\infty$$  Without loss of generality, we may assume that $t_n^k<k$ for all $n$.  

As we just saw, however, for $x\in N_A\cap S_{\lambda}$ and $|\lambda|$ sufficiently small, we have that $\omega_{S_\lambda}(x)\subset A_{\lambda}$.  Since $y_n\in \omega_{\lambda_n}(\psi_n^k(t_n^k))$ for any $k$ or $n$, we therefore must have that $$\psi_{n}^k(t_n^k)\in S_{\lambda_n}\setminus N_A$$ for all $n$ and $k$.  

By lemma \ref{extunisol}, for each $k$ there is an orbit $\psi^k$ on $\Phi_0$ such that on any compact interval, $\{\psi_n^k\}_{n=1}^\infty$ has some subsequence which converges uniformly to $\psi^k$.  Taking further subsequences if necessary, we also find limit points $t^k$ of $\{t_n^k\}_{n=1}^\infty\subset[0,k]$, $x$ of $\{x_n\}_{n=1}^\infty\subset N$, and $y^k$ of $\{\psi_n^k(t_n^k)\}_{n=1}^\infty\subset N$.  Notice that $x=\psi^k(0)$ and that $y^k=\psi^k(t^k)$.  Additionally, note that lemma \ref{inv_semi} implies that $x\in \cl(S_0\setminus N_R; S_0)$ and $y^k\in \cl(S_0\setminus N_A; S_0)$.

Since $y^k\in \cl(S_0\setminus N_A; S_0)$, there is some convergent subsequence $$y^{k_m}\to y\in \cl(S_0\setminus N_A; S_0)$$as $k_m\to\infty$.  That is, $$\psi^{k_m}(t^{k_m})\to y\in \cl(S_0\setminus N_A; S_0)$$This implies that $\omega_{S_0}(x)\not\subset A_0$, even though $x\in S_0\setminus R_0$, contradicting our assumption that $(A_0,R_0)$ is an attractor-repeller decomposition of $S_0$.  Therefore $\omega_{S_\lambda}(x)\subset A_\lambda$ for all $x\in S_\lambda\setminus R_\lambda$ for small enough $|\lambda|$, and attractor-repeller decompositions continue.  
\end{proof}

\section{Conclusions and Acknowledgements}
\subsection{Conclusions and Future Work}

The steps in this paper bring us one step closer to generalizing the Conley index theory to the setting of differential inclusions.  The next steps down this path will be to define the generalized Morse decomposition of these objects, and eventually the chain recurrent set.  Each of these tasks should follow fairly naturally from the attractor-repeller decomposition described in this paper.  

Still, there is a lot of work to be done before this generalization is completed.  One significant gap in this setting is the lack of a Lyapunov function for the multiflow.  In the classical attractor-repeller decomposition, there is a Lyapunov function which decreases on the connecting region $C(A,R)$.  While examples indicate that such a function does exist in general for multiflows, proving its existence is much more difficult because of the non-uniqueness of the solutions.  However, a similar result, in a slightly more restricted setting, is proven in \cite{kopanskii}, which demonstrates that the lack of uniqueness is not an insurmountable barrier.  Additionally, the actual index itself has not yet been generalized to differential inclusions, and this step will probably also be very difficult; again, however, such an index has been defined in a slightly more restricted, but still non-unique, setting (\cite{mrozek}).  The results in this paper and others do seem to indicate that this generalization is possible and that Conley index theory can be a useful tool for studying differential inclusions.  

\subsection{Acknowledgements} 
The ideas in this paper owe an enormous debt to Richard Moeckel.  His explanations of the classical Conley index theory were extremely illuminating and our conversations were incredibly valuable.  Additionally, we would like to thank all of the participants of the University of Minnesota's Mathematics of Climate Seminar for listening to an enormous number of presentations about these ideas.  Their comments and questions were crucial in the development of many of these ideas.

{\sc Cameron Thieme}

Department of Mathematics

University of Minnesota

206 Church Street Se

Minneapolis, Minnesota, US

Email: {\it thiem019@umn.edu}

\end{document}